\documentclass[11pt]{article}
\usepackage{mathrsfs}
\usepackage{amsfonts}
\usepackage{}
\usepackage{amsmath,amssymb,amsthm,latexsym,amstext}
\usepackage[mathscr]{eucal}
\usepackage{rotate,graphics,epsfig,epstopdf}
\usepackage{float}
\usepackage{color}
\usepackage{subfigure}
\usepackage{indentfirst}
\usepackage{bm}

\newcommand{\be}{\begin{eqnarray}}
\newcommand{\ee}{\end{eqnarray}}
\newcommand{\by}{\begin{eqnarray*}}
\newcommand{\ey}{\end{eqnarray*}}
\newcommand{\bn}{\begin{enumerate}}
\newcommand{\en}{\end{enumerate}}
\newcommand{\ei}{\end{itemize}}

\newtheorem{theorem}{Theorem}
\newtheorem{lemma}[theorem]{Lemma}

\newtheorem{corollary}{Corollary}[section]
\newtheorem{remark}[theorem]{Remark}
\newtheorem{definition}[theorem]{Definition}

\renewcommand{\theequation}{\arabic{section}.\arabic{equation}}

\numberwithin{equation}{section}

\begin{document}
\date{}
\title{\bf On uniqueness of solutions to degenerate nonlinear  Fokker-Planck Equations in Hilbert spaces  \footnote{This work was supported by
the National Natural Science Foundation of China  NSFC No.11771207.
}}
\author{ Xueru Liu\footnote{Corresponding author, 327625941@qq.com }\hskip1cm Xuan Yang\footnote{yangx38800@gmail.com} \hskip1cm Wei Wang\footnote{wangweinju@nju.edu.cn} \\
\texttt{{\scriptsize Department of Mathematics, Nanjing University,
Nanjing, 210023, P. R. China}}}\maketitle
\begin{abstract}
An ${ L^2({\mathbb{R}}^d)}$-valued stochastic  $N$-interacting particle systems  is investigated.   Existence and uniqueness  of  solutions for  the degenerate  nonlinear Fokker-Planck equation for probability measures  that  corresponds to  the mean field limit equation are derived.
\end{abstract}



\section{Introduction}
  \setcounter{equation}{0}
  \renewcommand{\theequation}
{1.\arabic{equation}}
In this paper, we study the following kinetic nonlinear Vlasov-Fokker-Planck equation  on a separable Hilbert space:
\begin{equation}\label{e:MFSPDE}
 {d}\mu_t + v\cdot \nabla_u \mu_t  {d}t= \tfrac{{1}}{\epsilon} \nabla_v \cdot (\gamma v- \triangle u )
\mu_t  {d}t-\tfrac{{1}}{\epsilon} \nabla_v \cdot (F(u, \rho_t)\mu_t {d}t +\tfrac{{1}}{2 \epsilon^2 }Tr(\sigma(u)\sigma^\ast(u)) \triangle_v \mu_t{d}t,
\end{equation}
where  $(\mu_t)_{t\geq 0}$ is a family of probability measures on $H \times H$. We denote by $H= L^2(\mathbb{R}^d)$ the Hilbert space.    Here,  $\rho_t =\int_{H} {d}\mu_t(\cdot ,v)$  denotes the $u$-marginal of $\mu_t$, The constant $\gamma>0$ is the frictional coefficient and the constant $\epsilon $ is  small mass, and
 $F: H \times\mathcal{P}(H) \rightarrow H $ is the driving force of the system,  which arises from an external or interaction potential. In typical applications,
 we assume $F$ has the following structure:
 $$F(u, \rho_t) = (\nabla \Psi)(u)+ (K\star \rho)(u), \\\  (u, \rho_t)\in   H \times\mathcal{P}(H) ,$$
 where  $(K\star \rho)(u)=\int_{H}K(u-u_1) {d}\rho_t(u_1)$,   and functions  $K: H\rightarrow H$, $\nabla \Psi : H\rightarrow H$ are uniformly Lipschitz continuous.
This structure corresponds to the Kolmogorov equation for a nonlinear stochastic differential equation
 \begin{eqnarray}
&&  {d}u_t  = v_t   {d}t, \label{e:SLE-11} \\
&& \epsilon  {d}v_t  = \triangle u_t    {d}t  -\gamma v_t  {d}t + F(u, \rho_t){d}t +
\sigma(u_t) {d}W_t .\label{e:SLE-2}
\end{eqnarray}
Write $\triangle$ the Laplacian on a Hilbert space $H $ with a domain $D(\triangle)$ and   $D(\triangle)= H_0^{1,2} \subset H$.
We denote the space of Hilbert-Schmidt operators $H \rightarrow H$ by $L_2(H,H)$, endowed the inner product $\langle A, B\rangle_{L_2(H,H)}= Tr_H[A^\ast B]= Tr_H[BA^\ast]$.
 Function  $\sigma :H \rightarrow L_2(H,H) $ is uniformly Lipschitz continuous and $W_t$ is  standard cylindrical  Wiener process on $H$, defined on a completed  probability basis space $(\Omega,\mathcal{F},\{ \mathcal{F}_t   \}_{t\geqslant 0},\mathbb{P} )$.

\noindent {{\bf Motivation of (\ref{e:MFSPDE}) from interacting particle systems:}} The kinetic nonlinear Vlasov-Fokker-Planck equation   (\ref{e:MFSPDE}) is closely
related to classical Newton dynamics for  $N$-interacting particle systems. More precisely, under suitable assumptions on $F$ and $\sigma$, (\ref{e:MFSPDE})  can be
derived from the following system of stochastic differential equations
\begin{eqnarray}\label{e:X-eps-1}
 {d}u_t^{i,N} &=& v_t^{i,N}  {d}t ,  \quad i=1,2,3,\cdot\cdot\cdot,N  \label{e:X-eps-1} \\
\epsilon  {d}v_t^{i,N} &=& \triangle u_t^{i,  N}  {d}t
-\gamma v_t^{i, N}  {d}t + \frac{ {1}}{N}\sum_{j=1}^N K( u_t^{i,N}- u_t^{j,N}) dt + (\nabla \Psi)(u_t^{i,N}) dt +
\sigma(u_t^{i,N})  {d}W_t^i . \label{e:X-eps-2}
\end{eqnarray}
Here, $u^{i, N}_t$ is the position of particle $i$ at time $t$.   $(W_t^1, W_t^2..., W_t^N)$ be $N$-independent standard cylindrical  Wiener process on $H$, defined on a completed  probability basis space $(\Omega,\mathcal{F},\{ \mathcal{F}_t   \}_{t\geqslant 0},\mathbb{P} )$.  By considering the mean-field limit  $N\rightarrow\infty$, the so-called nonlinear  Mckean-Vlasov stochastic differential equation (\ref{e:SLE-11})-(\ref{e:SLE-2}) replaces the system (\ref{e:X-eps-1})-(\ref{e:X-eps-2}), where
$\rho_t= Law(u_t)$ is the law of $u_t$, and $\mu_t=Law(u_t, v_t)$ satisfies  (\ref{e:MFSPDE})   in the sense of distributions. There has been a surge of activity for  stochastic $N$-particle system  of research  in finite dimensional space(\cite{AGBG},\cite{david},\cite{eweinan}).
It is particularly worth mentioning that
Liu and  Wang\cite{lw} consider the   interacting particles system (\ref{e:X-eps-1})-(\ref{e:X-eps-2}) with small mass in ${ L^2({\mathbb{R}}^d)}$. For fixed $\epsilon$, they prove that the solution to $(\ref{e:X-eps-2})$ converges to  that  $(\ref{e:SLE-2})$ uniformly for small mass $\epsilon$ of in the following sense
$$\lim_{N \rightarrow \infty}\mathbb{E} \|u^{i, N, \epsilon}_t-u_t^{i, \epsilon }\|^2_H =0. $$
In this paper, we show the limit of the following statistical quantities   given by the empirical measure
\begin{equation*}
 \Gamma_t^{N}:= \frac{ {1}}{N}\sum_{i=1}^N \delta_{(u_t^{i,N} ,v_t^{i,N}) }\,,
\end{equation*}
as $N\rightarrow\infty$\, by showing the well-posedness of the (\ref{e:MFSPDE}). Suppose that the empirical measure $ \Gamma_0^{N,\epsilon}:= \frac{\mathrm{1}}{N}\sum_{i=1}^N \delta_{u_0^{i} ,v_0^{i} } $
converges to a random probability measure $\mu_0$  in the metric $\mathbb{E}[W_1(\cdot,\cdot)]$, where $W_1$ is  1-Wasserstein metric, seeing the Definition \ref{def:P1} in section 2.

(\ref{e:MFSPDE}) is also called nonlinear Fokker-Planck equation{\cite{G.DP}} on infinite dimensional space $H_0^{1,2}\times H$. Nonlinear Fokker-Planck equations have been studied  in a variety of finite dimensional space. Papers by  McKean(\cite{Mh.p.1},\cite{Mh.p.2})  concerned with nonlinear parabolic equations. Such equations and the well-posedness of the martingale problem were studied by Funaki\cite{Funaki}.  Physical problems relating to nonlinear Fokker-Planck equations can be found in  \cite{Frank} and \cite{Bogachev}. Existence and uniqueness of solutions for such equations for measures were investigated(\cite{Manita3}, \cite {Manita2}).   In infinite dimensional case, Cauchy problem for the nonlinear Fokker-Planck-Kolmogorov equations for probability measures was studied by on a Hilbert space\cite{Manita1}. The work\cite{VIG} established the existence of solutions for nonlinear evolution equations for measures. For  interacting system,  Bhatt\cite{AGBG} studied such equations by solving martingale problems corresponding McKean-Vlasov equation on Hilbert spaces. In all the aforementioned papers, the nonlinear Fokker-Planck equations are non-degenerate. In our paper, we deal with the nonlinear Fokker-Planck equations(\ref{e:MFSPDE}), which is degenerate. We use the classical Holmgren method(\cite{Manita1},\cite{Manita2}) to show the uniqueness.

The rest of this paper is organized as follows. Some notations, assumptions and definition are introduced by Section 2.
In Section 3, we show the  well-posedness for the nonlinear Fokker-Planck equation(\ref{e:MFSPDE}).

\section{Preliminary}
  \setcounter{equation}{0}
  \renewcommand{\theequation}
{2.\arabic{equation}}

Let $\{  e_i\}_{i\in N}\subset H^2$ be the complete orthogonal basis of $H$.  Let $P_N$ be the orthogonal projection of $H$
onto $H_N= span\{e_1,...,e_N\}\cong \mathbb{R}^N$, For every $u\in H$, let $u_N$ denote the orthogonal projection of $u$ to $\mathbb{R}^N$, i.e., $u_N= P_Nu$.
Suppose that constants  are change during the proof of the result. Let $\mathcal{H}= H \times H $, and $\{  \hat{e}_i\}_{i\in N}$ be the complete orthogonal basis of $\mathcal{H}$. Now, we  introduce the usual test function space $\mathcal{F}C_0^\infty (\mathcal{H})$\cite{ZMZa}  on $\mathcal{H}$ consisting of finitely based smooth bounded functions,
$$\mathcal{F}C_0^\infty (\mathcal{H}):= \{\psi(l_1,...l_m) \mid   l_1,...l_m \in \mathcal{H}, \psi\in C_0^\infty(\mathbb{R}^{2m}) \}.$$

\begin{definition}\label{def:P1}
The metric space  $(\mathcal{P}_1(\mathcal{H}) ,W_1) $ is the space of probability measure $\mu(\cdot)$ on $\mathcal{H}$ with finite 1-moment, that is,
\begin{equation*}
 \int_{\mathcal{H}} {d}\mu(z)=1  , \quad   M_1( \mu):= \int_{\mathcal{H}}|z|  {d}\mu(z)<\infty\,,
\end{equation*}
endowed with the 1-Wasserstein metric
\begin{equation*}
W_1(\mu,\nu):= \sup \{\int f(z)(\mu-\nu)(dz): f \in \mathcal{F}C_0^\infty(\mathcal{H}), |\nabla f| \leq 1  \}.
\end{equation*}
\end{definition}

Let $Z_t =(u_t,  v_t)$,  $\tilde{A}Z_t= 1/\epsilon (v_t, \triangle u_t -\gamma v_t )  $
$\tilde{B}(Z_t, \mu_t)= 1/\epsilon (0, F(u, \rho_t))  $, $\tilde{\sigma}(Z_t)= 1/\epsilon (0, \sigma(u_t))$.
In this paper,  without loss of generality,  we take $\epsilon$=1. Then, the equation~(\ref{e:SLE-11})--(\ref{e:SLE-2}) is equivalent to the following equation
$${d}Z_t= \tilde{A}Z_t {d}t + \tilde{B}(Z_t, \mu_t)  {d}t +\tilde{\sigma}(Z_t) {d}W_t.$$
Let $\phi \in C_b^2({H}_0^{1,2} \times H)$, set
\begin{eqnarray}
L_{\mu}\phi &=& Tr(\tilde{Q}(z)D^2\phi ) +  \langle   \tilde{A}z+ \tilde{B}(z,\mu) , D\phi  \rangle , \nonumber\\
&:=&  \sum_{i,j=1}^{\infty}\tilde{a}^{ij}(z)\partial^2_{\hat{e}_i\hat{e}_j}\phi +  \sum_{i=1}^{\infty}\tilde{B}^i(z,\mu)  \partial_{\hat{e}_i}\phi +  \sum_{i=1}^{\infty} \tilde{A}^iz\partial_{\hat{e}_i}\phi,
\end{eqnarray}
here $z=(u,v)$, $ D\psi=(D_u\psi, D_v\psi)$, $D^2\psi =(\triangle_u\psi, \triangle_v\psi)$, $ \tilde{Q}(z)=(0, 1/2Tr(\sigma(u)\sigma^\ast(u)))$.
For fixed $\psi \in \mathcal{F}C_0^\infty(\mathcal{H})$,  that is $\psi \in C_0^\infty(\mathbb{R}^{2m})$, then
$$ L_{\mu}^m\psi = \sum_{i,j=1}^{2m}\tilde{a}^{ij}(z)\partial^2_{\hat{e}_i\hat{e}_j}\psi +  \sum_{i=1}^{2m}\tilde{B}^i(z,\mu)  \partial_{\hat{e}_i}\psi +  \sum_{i=1}^{2m} \tilde{A}^iz\partial_{\hat{e}_i}\psi . $$
Hence the nonlinear  Fokker-Planck equation (\ref{e:MFSPDE}) is equivalent to the following equation
\begin{eqnarray}\label{FPE-1}
\partial_t \mu_t + \nabla_z \cdot (  \tilde{A}z+    \tilde{B}(z,\mu_t))\mu_t =  Tr(\tilde{Q}(z)\triangle\mu_t),
\end{eqnarray}
here  $ \nabla_z \cdot \tilde{A}z \mu_t=     v\cdot \nabla_u \mu_t- \nabla_v \cdot (\gamma v- \triangle u )
\mu_t  {d}t $,   $ \nabla_z \cdot \tilde{B}(z,\mu_t) = \nabla_v \cdot (F(u, \rho_t) \mu_t  $,   and $Tr(\tilde{Q}(z)\triangle\mu_t) =   1/2Tr(\sigma(u)\sigma^\ast(u))\triangle_v\psi$.  \\
For each $N \in \mathbb{N}$, let $\tilde{A}_Nz = P_N\tilde{A}z := \{\tilde{A}^iz\}_{1\leq i\leq N}$, $\tilde{B}_N(z,\mu)= P_N\tilde{B}(z,\mu) :=\{\tilde{B}^{i}(z,\mu)\}_{1 \leq {i} \leq N} $ and $\hat{P}_N \tilde{Q}(z)= (0, Q_N(u))=(0, a^{i,j}(u))_{1 \leq {i,j} \leq N} := (\tilde{a}^{ij}(z))_{1 \leq {i,j} \leq N}$.

Now, we introduce the following assumptions.\\
\noindent {\bf \textbf{H$_1$}}
(1) There exist constants $L_\sigma$ and $L$, such that for every $T>0$,
 $$\|\sigma(u_1)-\sigma(u_2)\|_{L_2(H,H)}  \leqslant  L_\sigma \|u_1-u_2\|_H,   \\\ \\\ \|\sigma(u)\|_{L_2(H,H)} \leqslant  L( 1+ \|u\|_H).$$
~~~~~(2)The operator $Q(u)= 1/2 \sigma(u)\sigma(u)^\ast$, for every $k \in \mathbb{N}$, the matrix $\hat{P}_kQ$ take out the $k \times k$ matrix from  $Q(u)$, and we write  $\hat{P}_kQ= Q_k(u)=(a^{i,j}(u))_{1 \leq {i,j} \leq k}$, which is  symmetric and nonnegative definite.  $Q_k(u)$ has uniformly bounded elements with uniformly bounded first derivatives. Moreover, it is strictly elliptic: there exists $\theta$ such that for every $k \in \mathbb{N}$, $u\in  {H}$, $\langle Q_k(u)\xi, \xi  \rangle \geq \theta |\xi|^2,$ for all $\xi \in\mathbb{R}^{k}$.

 \noindent {\bf \textbf{H$_2$}} There exist constants $L_K$ and $K$, such that
 $$\|K(u_1)-K(u_2)\|_{H}  \leqslant  L_K\|u_1-u_2\|_H,   \\\ \\\ \|K(u)\|_{H} \leqslant  K( 1+ \|u\|_H).$$

 \noindent {\bf \textbf{H$_3$}} There exist constants $L_\Psi$ and $\hat{L}_\Psi$, such that
 $$\|\Psi(u_1)-\Psi(u_2)\|_{H}  \leqslant  L_\Psi\|u_1-u_2\|_H,   \\\ \\\ \|\Psi(u)\|_{H} \leqslant  \hat{L}_\Psi( 1+ \|u\|_H).$$

\begin{remark}
For every  $\mu \in P_1(\mathcal{H})$, there exists a constant $\alpha$ such that, for all  $z_1, z_2 \in \mathcal{H}$ and $t \in[0,T]$,
$$\langle  \tilde{B}(z_1,\mu ) -\tilde{B}(z_2,\mu ) , z_1-z_2 \rangle \leq \alpha|z_1-z_2|^2.$$
\end{remark}

\begin{definition}\label{def-1}
We say that $\mu_t = (\mu_t)_{t\in[0,T]}$ is a solution to the equation~(\ref{FPE-1}), if for every ${t\in[0,T]}$ and  $\varphi \in \mathcal{F}C_0^\infty (\mathcal{H})$,
$$ \int \varphi  {d}\mu_t = \int \varphi  {d}\mu_0 + \int_0^t \int L_{\mu}\varphi  {d}\mu_s  {d}s .$$
Sometime it is convenient to use an equivalent definition(see~\cite{Manita1}), assume that a test function $\Phi$ depends on a finite set of variables $z_1, z_2,..., z_m$, vanishes outside some ball in $H_m \oplus H_m  \cong \mathbb{R}^{2m},$  and $\Phi \in C^{2,1}(\mathbb{R}^{2m}\times (0,T))\cap C(\mathbb{R}^{2m}\times [0,T))$, for every ${t\in[0,T]}$
\begin{eqnarray}\label{DEF-1}
\int \Phi (z,t) {d}\mu_t =  \int \Phi (z,0) {d}\mu_0 + \int[\partial_s\Phi+ L_{\mu}\Phi ]  {d}\mu_s  {d}s.
 \end{eqnarray}
\end{definition}
Given a continuous strictly positive function $V = 1+ |Z|^2$ on $\mathcal{H}$, and $T>0$. Define
$$M_T(V):= \{ \mu =(\mu_t)_{t\in[0,T]}\in \mathcal{P}_1(\mathcal{H}): \sup_{t\in[0,T]}\int V(Z) {d}\mu_t(Z)< +\infty \}.$$
Then for all $\mu\in M_T(V)$ and $Z \in  \mathcal{H} $, there are constant $\Lambda_1$ and $\Lambda_1$ such that
$$L_\mu V(Z,t)\leq \Lambda_1+\Lambda_2 V(Z).$$
We say that a sequence $\mu^n= (\mu_t^n)_{t\in[0,T]}$ from the class $M_T(V)$ is V-convergent to a measure $\mu_t$ if for all $t\in[0,T]$
$$\lim_{n\rightarrow \infty }\int F(Z) {d}\mu_t^n(Z)= \int F(Z) {d}\mu_t(Z),$$
for every $F(Z)\in C(\mathcal{H})$, and such that $$\lim_{R\rightarrow \infty} \sup_{Z\in \mathcal{H} \setminus {B_R}} F(Z)\cdot V^{-1}(Z)=0,$$
here, $B_R = \{z \big|  \|Z\|_\mathcal{H}<R \}$. Obviously, if a sequence $\mu_t^n$ is weakly convergent,  it is V-convergent.

\begin{remark}\label{remark-n}
For fixed  $T>0$, the function  $ \tilde{B}(z_t,\mu)$ is Borel measurable on $t \in[0,T]$, and for every cylinder $\hat{H} \subset \mathcal{H}$ with a compact finite dimensional base, the function $ \tilde{B}(z_t,\mu)$   is  bounded on $\hat{H}$ uniformly in  $\mu \in  M_T(V)$ and $t\in [0,T]$. Moreover, if a sequence $\mu_t^n \in M_T(V)$ is V-convergent to a measure $\mu_t \in M_T(V)$. Then, for all $z= (u,v) \in \mathcal{H}, t\in [0,T],$
$$ \lim_{n\rightarrow \infty }\int_{\mathcal{H}}K(u-u_1) {d}\mu_t^n(u_1,v) = \int_{\mathcal{H}} K(u-u_1) {d}\mu_t(u_1,v).$$
\end{remark}

\section{Nonlinear Fokker-Planck Equations: Well-posedness}
  \setcounter{equation}{0}
  \renewcommand{\theequation}
{3.\arabic{equation}}

In this section, we show the existence and uniqueness  of the nonlinear Fokker-Planck equation(\ref{e:MFSPDE}).


\begin{lemma}\label{thm:main1}
Given $T>0$. Assume  {\bf H$_1$}-{\bf H$_3$} hold.  The nonlinear Fokker-Planck equation (\ref{FPE-1}) has a solution $(\mu)_{t\in[0,T]}\in M_T(V)$ in the sense of
Definition \ref{def-1}.
\end{lemma}

\begin{proof}
We construct a solution to  (\ref{FPE-1}) as a certain limit of solution to finite dimensional problems.
for each $N \in \mathbb{N},$  consider
$$\hat{Q}_N: z \rightarrow (\tilde{a}^{ij}(P_Nz))_{1 \leq {i,j} \leq N},$$
and
$$  \hat{A}_N: z \rightarrow (\tilde{A}^iP_Nz)_{1 \leq {i} \leq N}, \\\ \hat{B}_N: (z,\mu)\rightarrow ( \tilde{B}^{i}(P_Nz,\mu))_{1 \leq {i} \leq N},$$
here $P_Nz=\{z_1,...,z_N  \}$.  Let $L_\mu^N =  (\hat{A}_N+\hat{B}_N)\partial_{z_N} + Tr\hat{Q}_N \partial^2_{z_N}$, $z_N = P_Nz $, then the finite dimensional Fokker-Planck equation
\begin{equation}\label{FPE-2}
\partial_t \mu_t + \nabla \cdot ( \hat{A}_Nz+\hat{B}_N(z,\mu_t))\mu_t =  Tr( \hat{Q}_N \triangle\mu_t), \\\ \\\ \mu_0^N= \mu_0\circ P_N^{-1}.
\end{equation}
has a solution $\mu^N= (\mu_t^N)_{t\in[0,T]}$\cite{YOUNG-PIL}. We consider solution $(\mu_t^N)_{t\in[0,T]}$ as
measures on $\mathcal{H}$, let  $\mu_t^N(U\times V)=0 $ for every $U\subset \mathbb{R}^{2N}$ and nonempty $V \subset H \setminus {\mathbb{R}^{2N}}$.

Fix a function $\varphi(z) = \varphi(z_1, z_2 , ... , z_m )\in \mathcal{F}C_0^\infty(\mathcal{H}),$ and it has compact support $S\subset \mathbb{R}^{2m}$. For every $N\geq m $,
\begin{equation}\label{FPE-3}
\int_S \varphi  {d}\mu_t^N - \int_S \varphi  {d}\mu_0^N = \int_0^t \int_S L_\mu^N  \varphi  {d}\mu_s^N  {d}s ,
\end{equation}
and
$$| \int_S \varphi  {d}\mu_t^d- \int_S \varphi  {d}\mu_s^d| \leq C(\Lambda_1,  \Lambda_2 , \varphi)|t-s|.$$
Hence there exists a subsequence such that $\mu_t^{n_k}$ is a V-convergent to $\mu_t$ on $\mathcal{H}\times [0,T]$ as $k\rightarrow \infty$. Moreover,  $\mu_t^{n_k}$ converges weakly to $\mu_t$ for all $t\in[0,T]$, and $\mu_0^{n_k}$ converges weakly to   $\mu_0$. That is
$$ \int \varphi  {d}\mu_t^{n_k}\rightarrow \int \varphi  {d}\mu_t, \\\ \\\
\int \varphi  {d}\mu_0^{n_k}\rightarrow \int \varphi  {d}\mu_0. $$
Notice that Remark {\ref{remark-n}}, then by the Arzel\`{a}-Ascoli theorem, the sequences $B^{i}(z, \mu^{n_k})$ uniformly converge to  $B^{i}(z ,\mu)$ on compact sets in  $\mathcal{H}\times [0,T]$.
Clearly,
\begin{eqnarray}\label{FPE-4}
| \int_0^t \int L_\mu^{n_k}  \varphi {d}\mu_s^{n_k}   {d}s - \int_0^t \int L_\mu  \varphi {d}\mu_s {d}s|
 &\leq&|\int_0^t \int_S( L_\mu^{n_k}  \varphi  -  L_\mu  \varphi ){d}\mu_s^{n_k} {d}s| \nonumber\\
 &+& |\int_0^t \int_S L_\mu  \varphi {d}\mu_s^{n_k}  {d}s- \int_0^t \int_S L_\mu  \varphi {d}\mu_s {d}s|.
\end{eqnarray}
For (\ref{FPE-4}), by the uniform convergence of the coefficients, the first term on the right side tends to zero. On the other hand, $\mu_t^{n_k}(dz)$ converges weakly to $\mu_t(dz)$ for all $t\in[0,T]$, the second terms on the right side tends to zero.\\
Therefore, replacing $N$ by $n_k$ for (\ref{FPE-3}),  taking limit as $k\rightarrow +\infty,$ then
$$\int \varphi d \mu_t - \int \varphi d \mu_0  = \int_0^t \int L_\mu  \varphi d\mu_s {d}s .$$
The proof is complete.
\end{proof}

\begin{theorem}\label{thm:main2}

Given $T>0$. Assume  {\bf H$_1$}-{\bf H$_3$} hold.  Then the Fokker-Planck equation (\ref{FPE-1}) has a unique solution $(\mu_t)_{t\in[0,T]}$ in the sense of
Definition \ref{def-1}.

\end{theorem}

\begin{proof}
Assume that $(\mu_t)_{t\in[0,T]}\in M_T(V)$ and  $(\nu_t)_{t\in[0,T]}\in M_T(V)$ are solutions to (\ref{FPE-1}) with initial conditions $\mu_0\in \mathcal{P}_1(\mathcal{H})$ and $\nu_0\in \mathcal{P}_1(\mathcal{H})$ respectively. Fix a function $\psi_0 \in  \mathcal{F}C_0^\infty(\mathcal{H})$ such that $|\nabla \psi_0(z)|\leq1.$  Fix $N\in \mathbb{N}$ such that $ \psi(z)= \psi_0(P_Nz)$.
Notice that $\tilde{B}_N(z,\mu)= P_N\tilde{B}(z,\mu)$ and $\tilde{B}(z,\mu)=(0, \int_HK(u-u_1)\mu(u_1,v))$,  then fix $\varepsilon >0$, by {\bf H$_2$} and {\bf H$_2$}, there exists a smooth finite dimensional approximating sequence  $\hat{B}_{\mu, N}\in C^\infty(\mathbb{R}^{2N}, [0,T])$ such that for every $\nu \in M_T(V)$, we have $\hat{B}_{\mu, N}\in L^1(\mathcal{H}, \mu+\nu) $, and
\begin{eqnarray}\label{F-1}
\int_0^T \int_{\mathcal{H}} |\tilde{B}_{N}(z_t,\mu )-\hat{B}_{\mu, N}(P_N z_t)|(\mu_t+\nu_t)dz dt < \varepsilon.
\end{eqnarray}
Similarly, let  $\hat{A}_N:  z \rightarrow (\tilde{A}^iP_Nz_t)_{1 \leq {i} \leq N},$  $\hat{Q}_N: z \rightarrow ( \tilde{a}^{ij}(P_Nz_t))_{1 \leq {i,j} \leq N},$ then
$$\lim_{N\rightarrow \infty}\hat{A}_Nz = \tilde{A}z, \\\  \lim_{N\rightarrow \infty}\hat{Q}_Nz = \tilde{Q}z,$$
Fixed a function $\phi \in C_0^\infty(\mathbb{R}^1)$ such that $0\leq \phi (u)\leq1$ for $u \in \mathbb{R}^1$, and $\phi(u)=1$, for $|u|<1$,  and $\phi(u)=0$, for $|u|>2$, moreover, for all $u \in \mathbb{R}^1$, there exists a constant $C$, such that $|\phi''(u)|^2+|\phi'(u)|^2 \leq C\phi(u)$. For each $M>0$, set
$\phi_M(t,z) := \phi(t/M)\cdot \phi(|z|/M)$.
Now, we split several steps to prove the theorem.

Step 1. "The adjoint problem".  For $t\in [0,T]$, suppose $s\in [0,t]$, the equation
\begin{eqnarray}\label{UN-1}
\partial_s f_N + \hat{L}_{\mu} f_N = 0.   \quad  and \quad  f|_{s=t}= \psi,  \quad s\in[0,t],
\end{eqnarray}
with
$$\hat{L}_{\mu} f_N := Tr (\hat{Q}_N(z)D^2 f_N)  +  \langle \hat{A}_Nz+  \hat{B}_{\mu, N}(z) , D f_N  \rangle ,$$
has a solution $f_N $ in $\mathbb{R}^{2N}$,  and $f=f_N \in C^{2,1}(\mathbb{R}^{2N}\times[0,t])$. Indeed, the stochastic differential equation in $\mathbb{R}^{2N}$,
$$Z_t^N= \hat{A}_N Z_t^N dt + \hat{B}_{\mu,N}( Z_t^N) {d}t +\hat{\sigma}_N (Z_t^N) {d}W_t , \\\ Z_0^N=z ,$$
has a solution $Z_t^N$, $t \geq 0$, and the function $ f(s,z)= \mathbb{E}(\psi(Z_t^N) \big| Z_s^N = z)$ solves the (\ref{UN-1}). Moreover, $|f|\leq \max|\psi|:=C(\psi)$.

Step 2. let  $\Phi= \phi_M f$, then plugging $\Phi$ into (\ref{DEF-1}) for solution  $(\mu_t)_{t\in[0,T]}$,
$$\int \phi_M (t, z) \psi(z){d}\mu_t = \int \phi_M (0, z) f(0, z){d}\mu_0 + \int_0^t \int [\partial_s( \phi_M f) + L_{\mu}( \phi_M f) ]{d}\mu_s \mathrm{d}s$$
$$ L_{\mu}( \phi_M f) = Tr(\tilde{Q}(z)D^2( \phi_M f))  +   \langle \tilde{A}z+  \tilde{B}(z,\mu) , D( \phi_M f)  \rangle , $$
notice that $\partial_s f_N + \hat{L}_{\mu} f_N = 0$, then
$$\partial_s( \phi_M f) = (\partial_s \phi_M)  f + (\partial_s f) \phi_M =  (\partial_s \phi_M)  f + (-\hat{L}_{\mu} f) \phi_M $$
$$ =( \partial_s \phi_M)  f  - \phi_M ( Tr(\hat{Q}_N(z)D^2 f)  +  \langle \hat{A}_Nz+  \hat{B}_{\mu, N}(z)  ,  f \rangle),$$
since
$$D^2( \phi_M f)  = \nabla\cdot\nabla( \phi_M f)= \phi_M \cdot\triangle f + \triangle \phi_M\cdot f + 2 \nabla f \cdot\nabla \phi_M , $$
hence
\begin{eqnarray}\label{e-1}
&&\int \phi_M (t, z) \psi(z){d}\mu_t = \int \phi_M (0, z) f(0, z){d}\mu_0 + 2 \int_0^t \int  \langle  (Tr\tilde{Q}(z) )\nabla \phi_M  ,  \nabla f \rangle{d}\mu_s {d}s  \nonumber\\
&+& \int_0^t \int  \phi_M   \langle  \tilde{B}(z,\mu) -  \hat{B}_{\mu, N}(z)   ,   \nabla f \rangle {d}\mu_s {d}s + \int_0^t \int  \phi_M   \langle  \tilde{A}z -  \hat{A}_{ N}z  ,   \nabla f \rangle {d}\mu_s {d}s \nonumber\\
&+&  \int_0^t \int \phi_M  Tr(\tilde{Q}(z)- \hat{Q}_N(z))\triangle f  d\mu_s {d}s +   \int_0^t \int     f( \partial_s \phi_M )+ f L_\mu \phi_M  d\mu_s {d}s .
\end{eqnarray}
Similarly for  solution  $(\nu_t)_{t\in[0,T]}$, then
\begin{eqnarray}\label{e-2}
&&\int \phi_M (t, z) \psi(z){d}\nu_t = \int \phi_M (0, z) f(0, z){d}\nu_0 + 2 \int_0^t \int  \langle  (Tr \tilde{Q}(z)) \nabla \phi_M  ,  \nabla f \rangle{d}\nu_s {d}s  \nonumber\\
&+& \int_0^t \int  \phi_M   \langle  \tilde{B}(z,\nu) -  \hat{B}_{\mu, N}(z)  ,   \nabla f \rangle {d}\nu_s {d}s + \int_0^t \int  \phi_M   \langle  \tilde{A}z -  \hat{A}_{ N}z  ,   \nabla f \rangle {d}\nu_s {d}s \nonumber\\
&+&  \int_0^t \int \phi_M  Tr(\tilde{Q}(z)- \hat{Q}_N(z))\triangle f  d\nu_s {d}s +   \int_0^t \int     f( \partial_s \phi_M )+ f L_\mu \phi_M  d\nu_s {d}s .
\end{eqnarray}
Subtracting the equation (\ref{e-2}) from the equation (\ref{e-1}), then
\begin{eqnarray}
&& \int \phi_M (t, z) \psi(z) {d}(\mu_t - \nu_t) \leqslant \int | \phi_M f |{d}(\mu_0 - \nu_0)+ \int_0^t \int  |f ||L_\mu\phi_M | {d}(\mu_s +\nu_s){d}s  \nonumber\\
&+ &  \int_0^t \int \big[ \phi_M |  \tilde{B}(z,\mu) -  \hat{B}_{\mu,N}(z)  ||\nabla f |+   \phi_M   |\langle  \tilde{A}z -  \hat{A}_{ N}z  ,   \nabla f \rangle |\big]({d}\mu_s + {d}\nu_s) {d}s  \nonumber\\
&+&   \int_0^t \int  \phi_M | \tilde{B}(z,\nu) -  \tilde{B}(z,\mu)  ||\nabla f | {d}\nu {d}s + 2 \int_0^t \int  Tr(\tilde{Q}(z))|\nabla \phi_M | |\nabla f| {d}(\mu_s +\nu_s){d}s \nonumber\\
&+&  \int_0^t \int \phi_M  |Tr(\tilde{Q}(z)- \hat{Q}_N(z))\triangle f| (d\mu_s+d\nu_s) {d}s .
 \end{eqnarray}

Step 3. In this step, we show that $\nabla f$ is bounded. Using {\bf H$_1$}(2), there exists a constant $\varpi>0$, such that $| \nabla \hat{Q}_N(z) |\leq \varpi.$  Let  $G_{\mu,N}(z,t)= \hat{A}_Nz+  \hat{B}_{\mu,N}(z)$,  then there exists $\tilde{\alpha}$ such that
$$\langle \mathcal{G}(t,z)z', z'\rangle \leq  \tilde{\alpha}|z'|^2 \\\ where \\\ \mathcal{G}= (\partial_{z_j}G^i_{\mu,N})_{i,j\leq N} .$$
Now, let $\chi(t,z) = (\nabla f)^2 + \kappa   f^2 $, then
\begin{eqnarray}\label{e-3again}
-(\partial_s + \tilde{L}_{\mu})\chi & = & \nabla f ( \nabla  (Tr\hat{Q}_N(z)) ) \cdot  \triangle f ))+ 2\nabla f \langle \nabla G ,  \nabla f \rangle
- 2  (Tr\hat{Q}_N(z))  (\triangle f)^2 - 2\kappa   (Tr\hat{Q}_N(z))   (\nabla f)^2  \nonumber\\
&\leqslant & | \nabla  (Tr\hat{Q}_N(z))  |c^{-1}  (\nabla f)^2 + c  (\triangle f)^2 + 2\tilde{\alpha}(\nabla f)^2  - 2\theta (\triangle f)^2 - 2\kappa   \theta (\nabla f)^2  \nonumber\\
&\leqslant &  \varpi c^{-1}  (\nabla f)^2 + c  (\triangle f)^2 + 2\tilde{\alpha}(\nabla f)^2 - 2\theta (\triangle f)^2 - 2\kappa    \theta (\nabla f)^2,
 \end{eqnarray}
let $c=2\theta  $ and  $\kappa    = (\varpi c^{-1}  + 2\tilde{\alpha}) /(2\theta)$.
Then
$$-(\partial_s + \tilde{L}_{\mu} )\chi  \leqslant 0$$
Using the maximum principle\cite[Therorem 3.1.1]{DWS},
$$\max_{\mathbb{R}^{2N}\times[0,T]}|\chi(z,t) | \leqslant \max_{\mathbb{R}^{2N}}|\chi (z)| \leqslant \max_{\mathbb{R}^{2N}}( |\nabla \psi| ^2 + \kappa  |\psi|^2  ) ,$$
hence
$$\sup_{\mathbb{R}^{2N}\times[0,T]}|\nabla f| \leqslant [\max_{\mathbb{R}^{2N}}(|\nabla \psi| ^2 + \kappa  |\psi|^2 )]^{1/2} =: \tilde{C}.$$
Using equation~(\ref{e-3again}), we can easily obtain $\sup_{(z,t)\in\mathbb{R}^{2N}\times[0,T]}|\partial_{z_i}\partial_{z_j}f|\leq C(\psi)$, one can also see Theorem 2.8\cite{NVKE}.

Step 4. Taking limits as $M\rightarrow \infty$, $N\rightarrow \infty$.
By the maximum principle and the Arzel\`{a}-Ascoli theorem, the sequence ${f_N(0,z)}$ has a subsequence converging on compact sets. In particular, ${f_N(0,z)}\rightarrow \tilde{f}(z)\in C^1_b(\mathcal{H})$ as $N\rightarrow \infty.$ Then
$$\int | \phi_M  \tilde{f}|{d}(\mu_0 - \nu_0) \leqslant C_1W_1(\mu_0 , \nu_0). $$
Since $\nabla f$ and $\triangle f$ are bounded, then
$$\int_0^t\int  \phi_M | \tilde{B}(z,\mu) -  \hat{B}_{\mu,N}(z)  ||\nabla f | {d}(\mu_s +\nu_s){d}s  < \varepsilon ,$$
$$\int_0^t \int  \phi_M   \langle  \tilde{A}z -  \hat{A}_{ N}z  ,   \nabla f \rangle ({d}\mu_s + {d}\nu_s) {d}s  \rightarrow 0, \\\ as \\\ N \rightarrow \infty , $$
$$ \int_0^t \int \phi_M  |Tr(\tilde{Q}(z)- \hat{Q}_N(z))\triangle f| (d\mu_s+d\nu_s) {d}s\rightarrow 0, \\\ as \\\ N \rightarrow \infty .$$
Using  {\bf H$_2$} and  {\bf H$_3$}
\begin{eqnarray}
\int_0^t \int  \phi_M | \tilde{B}(z,\nu) - \tilde{B}(z,\mu)  ||\nabla f | {d}\nu {d}s \leqslant  C L_K \int_0^T  \int  W_1(\mu_s, \nu_s) {d}\nu_s{d}s
\end{eqnarray}
For $L_\mu\phi_M$,
$$|L_\mu\phi_M|\leq | \tilde{L}_\mu\phi_M | + |L_\mu\phi_M - \tilde{L}_\mu\phi_M | ,$$
using  the definition of $\phi_M$, then
$$\lim_{N\rightarrow \infty}|\langle  \tilde{B}(z,\mu)- \hat{B}_{\mu,N}(z), \nabla \phi_M\rangle| +|\langle  \tilde{A}(z)- \hat{A}_{N}(z), \nabla \phi_M\rangle|  =0,$$
and
$$\lim_{N\rightarrow \infty}|\langle Tr(  \tilde{Q}(z)- \hat{Q}_{N}(z) ), \triangle \phi_M\rangle| =0.$$
Notice that
$$\tilde{L}\phi_M = 1/M \phi_M'(|z|/M) \langle\hat{A}_{N}(z)+ \hat{B}_{\mu,N}(z) , z/|z| \rangle + 1/M^2 \phi_M''(|z|/M) Tr(\hat{Q}_N(z)), $$
for a fixed $N$,
$$|\tilde{L}\phi_M| \leq \big[1/M | \hat{A}_{N}(z) +\hat{B}_{\mu,N}(z)| + 1/M^2| Tr \hat{Q}_N(z)|\big]I_{\{M\leq |z| \leq 2M\} }\leq  CI_{\{M\leq |z| \leq 2M\} }.$$
Thus
$$2 \int_0^T \int  |f ||L_\mu\phi_M | {d}(\mu_s +\nu_s){d}s \rightarrow 0, \\\ as \\\ M\rightarrow \infty, N\rightarrow \infty  .$$
Similarly,
$$\int_0^T \int |Tr( \tilde{Q}(z) )\nabla \phi_M | |\nabla f| {d}(\mu_s +\nu_s){d}s \rightarrow 0, \\\ as \\\ M\rightarrow \infty . $$
Step 5. The estimation. Since $\varepsilon$ is an arbitrary number, from what have been proved, then
$$W_1(\mu_t , \nu_t) \leqslant  C\int_0^T W_1(\mu_s, \nu_s)\mathrm{d}s + CW_1(\mu_0, \nu_0), $$
using Gronwall inequality
$$W_1(\mu_t , \nu_t)\leqslant CW_1(\mu_0, \nu_0) . $$
The proof is completed. $\square$
\end{proof}

\begin{corollary}\label{mut_1}
Given $T > 0 $. Assume {\bf H$_1$}-{\bf H$_3$}  hold.  Let $( u_t^{i,  N}, v_t^{i,  N})$ and $( u_t, v_t)$ be  solutions of equations~(\ref{e:X-eps-1})--(\ref{e:X-eps-2}) and~(\ref{e:SLE-11})--(\ref{e:SLE-2}) respectively, For every $t\in [0,T]$,
$$\lim_{N\rightarrow \infty }\mathbb{E}\left[W_1(\Gamma_t^{ N}, \mu_t)\right]  =0.$$
\end{corollary}

\begin{proof}
Fixed a function $\varphi(z)=\varphi(z_1.z_2,...,z_m)\in \mathcal{F}C_0^\infty(\mathcal{H}),$ that is $\varphi(z)\in C_0^\infty(\mathbb{R}^{2m})$, then $\mu^m$ is
the solution of finite dimensional Fokker-Planck equation
\begin{equation}
\partial_t \mu_t + \nabla \cdot ( \tilde{A}_mz+\tilde{B}_m(z,\mu_t))\mu = Tr( \tilde{Q}_m \triangle \mu_t), \\\ \\\ \mu_0^m= \mu_0\circ P_m^{-1}.
\end{equation}
We define  $\Gamma_{t(m)}^{\epsilon, N} =  \frac{ {1}}{N}\sum_{i=1}^N \delta_{(u_{t(m)}^{i, N} ,v_{t(m)}^{i, N}) }$, here $u_{t(m)}^{i ,N}= P_m u_t^{i, N}$ and   $v_{t(m)}^{i, N}= P_m v_t^{i,   N},$
we replace $( u_t^{i,   N}, v_t^{i,  N})$ with $(u_{t(m)}^{i, N} ,v_{t(m)}^{i, N})$ for equations~(\ref{e:X-eps-1})--(\ref{e:X-eps-2}),
then we obtain an interacting particles system on $\mathbb{R}^{2m}$, by the Lemma 10 of \cite{Karl oe},
$$\lim_{N\rightarrow \infty }\mathbb{E}\left[W_1(\Gamma_{t(m)}^{ N}, \mu_{t(m)})\right]  =0.$$
Due to the arbitrariness of $m$, then
$$\lim_{N\rightarrow \infty }\mathbb{E}\left[W_1(\Gamma_t^{  N}, \mu_t)\right]  =0.$$

\end{proof}

\begin{remark}
We obtain the asymptotic behavior of the sequence of empirical measures $\Gamma_t^{ N}$ by showing the  existence and uniqueness of the corresponding  nonlinear Fokker-Planck equation. One of the similarly studied particle system  model which solves the nonlinear equation by the McKean-Vlasov martingale problem, introduced in \cite{AGBG}.

\end{remark}


\begin{thebibliography}{99}
\small \setlength{\itemsep}{-.8mm}




\bibitem{AGBG} A.G. Bhatt, G.Kallianpur, R.L. Karandikar and J.Xiong. On interacting systems of Hilbert-space-valued diffusions. {\em Applied Mathematics and Optimization}, {\bf 37(2)}, 151-188, 1998.


\bibitem{david} David Criens, Propagation of chaos for weakly interacting mild solutions to stochastic partial differential equations. arXiv.2017.13397v1[math.PR].




\bibitem{DWS} D.W.Stroock and S.R.S. Varadhan, {\em Multidimensional Diffusion Processes}, Springer, Berlin, 1979.
\bibitem{eweinan} E W, Shen H.  Mean field limit of a dynamical model of polymer systems. {\em Sci China Math}, {\bf56}: 2591-2598, 2013.




\bibitem{G.DP}G. Da. Prato and J. Zabczyk. {\em Stochastic Equations in Infinite Dimensions}, Cambridge Univ. Press, Cambridge, 2014.

\bibitem{Mh.p.1} H.P. McKean, A class of Markov processes associated with nonlinear parabolic equation.{\em Proc. Natl. Acad. Sci. USA}.  {\bf56}, 1907-1911, 1966.

\bibitem{Mh.p.2} H.P. McKean, Propagation of chaos for a class of non-linear parabolic equations.  In: {\em Lecture Series in Differential Equations},  Session, {\bf7},  177-194. Catholic University, 1967.
{\bf 103}, 143-158, 1996.

\bibitem{Karl oe} K. Oelschl\"{a}ger, A martingale approach to the law of large numbers for weakly interacting stochastic process. {\em The Annals of Probability}, {\bf 12}(2), 458-479, 1984.


\bibitem{lw} X.Liu, W. Wang, Small mass limit  for interacting particles system in ${\bf L^2({\mathbb{R}}^d)}$.

\bibitem{NVKE}N.V. Keylov, E.Priola, Elliptic and parabolic second-order PDEs with growing coefficient, {\em Comm. Partial Differential Equations}, {\bf 35}(1), 1-22, 2009.


\bibitem{Manita1} O.A. Manita,  Nonlinear Fokker-Planck-Kolmogorov equations in Hilbert space. {\em J. Math.  Sci.}  {\bf 216}(1), 120-135, 2016.
\bibitem{Manita3} O.A. Manita,  Romanov, M.S.,  Shaposhnikov, S.V. On uniqueness of solutions to noninear Fokker-Planck-Kolmogorov equations.  {\em  Nonlinear  Anal. Theory Methods Appl.} {\bf 128}, 199-226, 2015.
\bibitem{Manita2} O.A. Manita, S.V. Shaposhnikov,   Nonlinear    parabolic  equations for measures. {\em St. Petersb.  Math. J.} {\bf 25}, 43-62, 2014.
 \bibitem{Frank} T.D. Frank,  {\em Nonlinear Fokker-Planck Equations.} Fundamentals and Applications. Springer, Berlin, 2005.
\bibitem{Funaki} T. Funaki,   A certain class of diffusion processes associated with nonlinear parabolic equations. {\em Z. Wahrscheinlichkeitstheorie verw. Geb.} {\bf 67}, 331-348, 1984.
\bibitem{VIG} V.I. Bogachev, G. Da. Prato, M.R\"{o}ckner and S.V. Shaposhnikov,  Nonliner evolution equations for measures on infinite dimensional space, in:
{\em Stochastic Partial Differential Equations and Applications}, Quaderni di Matematica, {\bf 25}, 51-64, 2010.
\bibitem{Bogachev} V.I. Bogachev,  N.V. Krylov, M. R\"ockner,  Shaposhnikov, S.V.  {\em Fokker-Planck-Kolmogorov Equations. } American Mathematical Society,  Providence. Rhode Island, 2015.
\bibitem{YOUNG-PIL}Y.-P. Choi, O. Tse,  Quantified overdamped limit for kinetic Vlasov-Fokker-Planck equations with singularinteraction forces.  {\em Journal of Differential Equations.} {\bf 330}, 150-207, 2022.

\bibitem{ZMZa} Z.M.Ma, M.R\"{o}ckner. {\em An Introduction to The Theory of (Non-Symmetric) Dirichlet Forms}. Berlin, Springer 1992.





\end{thebibliography}
\end{document}